\documentclass[11pt]{amsart}
\usepackage{geometry}
\usepackage{amsmath}
\usepackage{amsthm}
\usepackage{amssymb}
\usepackage[utf8]{inputenc}
\usepackage[T1]{fontenc}

\usepackage{enumerate}
\usepackage{todonotes}
\usepackage{mathrsfs}
\usepackage{hyperref}
\usepackage{nicefrac}
\usepackage{tikz}
\usetikzlibrary{matrix}

\newcommand{\N}{\mathbb{N}}
\newcommand{\Z}{\mathbb{Z}}
\newcommand{\T}{\mathbb{T}}
\newcommand{\R}{\mathbb{R}}

\newcommand{\Q}{\mathbb{Q}}
 
\newtheorem{theorem}{Theorem}
\newtheorem{lemma}{Lemma}
\newtheorem{corollary}{Corollary}

\title{Weakly Mixing Systems with Dense Prime Orbits}
\author{Aaron Benda}
\date{July 2019}

\begin{document}

\maketitle

\begin{abstract}
   We show existence of smooth, weakly mixing reparametrizations of some linear flows on $\mathbb{T}^2$ for which all orbits sampled at prime times are dense.
\end{abstract}

\section{Introduction}
In this paper, we study orbits of a class of smooth dynamical systems sampled at prime times. More precisely, given a topological dynamical system $(X,T)$ we study the behavior of $\{T^px\}_{p \text{ prime}}$ for $x\in X$. In general, understanding prime orbits is a difficult task; there are only a few examples of dynamical systems for which all prime orbits are well understood (see \cite{Bourgain2}, \cite{Vinogradov}, \cite{Green-Tao}). In two recent papers (\cite{Kan} and \cite{KLR}), the authors developed new techniques for studying prime orbits for dynamical systems that are different than the classical methods (primarily that of type I and type II sums). Their methods are more focused on the periodic structure of the system, and provide the inspiration for techniques in this paper. In Corollary 1.2 in \cite{Kan}, Kanigowski established existence of smooth weakly mixing flows for which all prime orbits are equidistributed. In this Theorem, the class of flows are analytic reparametrizations of Liouvillean linear flows on $\T^2$. To get equidistribution, the author needs to construct the set of Liouvillean frequency implicitly. The results of this paper are similar, and some of the key estimates are in the same spirit as Theorem 1.3 from \cite{Kan} as well. The purpose of this note is to provide a much simpler argument for the existence of a class of such reparametrizations of some linear flows (with much milder assumptions on the frequencies than in \cite{Kan}) for which all prime orbits are {\em dense} (and in fact are equidistributed along a subsequence).

For an irrational $\alpha$ let $T^\alpha_t:\T^2\to \T^2$ be the linear flow in direction $\alpha$. For a (smooth) function $r:\T^2\to \R^+$, let $(T^{\alpha,r}_t)$ denote the {\em reparametrization} given by $r$ with the (unique) invariant measure $\mu_r$ (see section 3 for the definitions). Our main result is the following:

\begin{theorem}
For uncountably many irrational $\alpha$, there is a smooth $r: \T^2 \to \R^+$ such that the flow $(T^{\alpha,r}_t)$ on $(\T^2,\mu_r)$ is weakly mixing and there is an increasing sequence of numbers $(N_k) \to +\infty$ such that for all $x \in \T^2$, the sequence $\{T^{\alpha,r}_p(x) \mid p \text{ is prime}, p < N_k \}$ is equidistributed in $\T^2$.
\end{theorem}

As an immediate corollary, we obtain that the prime orbits of all points are dense:

\begin{corollary}
For all $x \in \T^2$, the set $\{ T^{\alpha,r}_p(x) \mid p \text{ is prime} \}$ is dense in $\T^2$.
\end{corollary}

As mentioned above, we prove a stronger statement: namely that there exists a sequence $N_k\to +\infty$ such that for  all $x \in \T^2$, the orbit $\{T^{\alpha,r}_p(x)\}_{p\leq N_k}$ is equidistributed with respect to $\mu_r$.
Our methods follow similar strategy to that in \cite{Kan} and \cite{KLR}, i.e. we show that there is a periodic structure of the orbits in their respective residue classes modulo sufficiently large primes, allowing us to approximate prime orbits by usual $\Z$ orbits and conclude by the use of unique ergodicity of the time $1$ map. We leverage more classical results from number theory (in particular the Siegel-Walfisz Theorem) than the ones used in \cite{Kan} (where the author uses the Bombieri-Vinogradov Theorem). On the dynamical side, the crucial part is to gain control of ergodic sums for the time change map $r$ along with fine approximation of the system by a periodic one.  Many of these estimates on the ergodic sums are in the same spirit as the Denjoy-Koksma inequality, and depend heavily on the imposed diophantine conditions on the irrational $\alpha$.

\section{Acknowledgements}
I owe Dr. Adam Kanigowski many thanks for his help in completing this project.
I would also like to thank the organizers of the Maryland Summer Scholars program for providing the funding and opportunity for this project to take place.

\section{Basic Definitions and Background}

\subsection{Notation}

Throughout the paper $(X,\mathscr{B},\mu)$ will denote a standard Borel probability space. Given a measure-preserving transformation $T: (X,\mathscr{B},\mu) \to (X,\mathscr{B},\mu)$ and a function $f: (X,\mathscr{B},\mu) \to \R^{+}$, we denote the ergodic sums by

$$
S_n(f)(x) = \sum_{j = 0}^{n-1} f(T^j x).
$$

Also, the notation $\Vert \cdot \Vert$ is used to denote the distance to the nearest integer throughout.

\subsection{Irrational Rotations of $\T$}

We will treat $\T$ as the interval $[0,1)$ with the points 0 and 1 identified. Let $\alpha$ be irrational. We define a transformation $T_{\alpha}: \T \to \T$ by $T_{\alpha}(x) = x + \alpha \pmod 1$. A classical fact is that these transformations are uniquely ergodic.

Every irrational $\alpha \in (0,1)$ has a unique continued fraction expansion, by which we mean there is a unique sequence of positive integers $(a_n)$ such that
$$
\alpha = \frac{1}{a_1 + \frac{1}{a_2 + \frac{1}{a_3 + ...}}}.
$$
Given an $\alpha$ with continued fraction expansion $(a_n)$, we define inductively the sequence of denominators $(q_n)$:
$$
q_0 = q_1 = 1, \text{ and } q_{n+1} = a_n \cdot q_n + q_{n-1}.
$$
In particular, there also exists a unique sequence of numerators $(p_n)$ (satisfying an analogous recurrence relation to the denominators) such that
$$
\frac{p_n}{q_n} = \frac{1}{a_1 + \frac{1}{a_2 + \frac{1}{... + \frac{1}{a_n}}}}.
$$
The primary important property of this sequence of rational approximations (which we will make use of repeatedly) is that for every $n \in \N$,
$$
\frac{1}{2q_n q_{n+1}} < \biggr\lvert \alpha - \frac{p_n}{q_n} \biggr\rvert < \frac{1}{q_n q_{n+1}},
$$
and this provides us with a very nice bound on the distance from $q_n \alpha$ to the nearest integer:
\begin{equation}\label{eq:norm}
    \frac{1}{2 q_{n+1}} < \lvert q_n \alpha - p_n \rvert < \frac{1}{q_{n+1}}.
\end{equation}
We recall that the $(p_n/q_n)$ are the best rational approximations to $\alpha$ in the following sense: if for some $r \leq q_n$ and $0 \leq s < r$ we have $\lvert \alpha - s/r \rvert \leq \lvert \alpha - p_n/q_n \rvert$, then $r = q_n$ and $s = p_n$.

\subsection{Special Flows}

Given a transformation $T: (X,\mathscr{B},\mu) \to (X,\mathscr{B},\mu)$ and an $L^1$ function $f: (X,\mathscr{B},\mu) \to \R^{+}$, consider the quotient space
$$
X_f = \{ (x,t) \in X \times \R^{+} \mid 0 \leq t \leq f(x) \} / \sim,
$$
where $\sim$ is the equivalence relation given by $(x, f(x)) \sim (T(x), 0)$. We define a flow $(T_t^f)$ on this space in the following manner: given a point $(x, s)$, flow vertically with unit speed until reaching the point $(x, f(x))$, then jump to $(Tx, 0)$ and flow until $(Tx, f(Tx))$, then jump to $(T^2 x, 0)$, and so on. Formally, it is defined by
$$
T_t^f (x, s) = (T^n x, s'),
$$
where $n, s'$ are uniquely defined and satisfy
$$
S_n(f)(x) + s' = t + s, \; \text{ and } \; 0 \leq s' \leq f(T^n x).
$$
Observe that $n, s'$ are unique since the function $f$ necessarily outputs positive numbers. 
These kinds of dynamical systems are known as \it special flows. \normalfont This is the kind of system we will consider. These flows preserve the product measure on the $\sigma$-Algebra generated by the given $\mathscr{B}$ and the Borel $\sigma$-Algebra on $\R$. It is clear that ergodic properties of $(T_t^f)$ depend heavily on ergodic properties of the underlying $T$; it is true, for example, that ergodicity of $T$ implies ergodicity of $(T_t^f)$.

\subsection{Smooth Flows on $\T^2$}

We define a flow $(T^{\alpha,r}_t)$ on the 2-Torus, $\T^2 = \T \times \T$ in the following manner: given an analytic function $r: \T^2 \to \R^{+}$ and an (irrational) angle $\alpha$, set $\frac{dy}{dt} = \alpha \cdot r(x,y)$ and $\frac{dx}{dt} = r(x,y)$. For a given point $(x,y) \in \T^2$, this essentially means we flow at a slope $\alpha$ from the horizontal, with speed scaled by our function $r$. As alluded to, this is simply a reparameterization of the linear flow in the direction $\alpha$.

This is a classical example of a smooth dynamical system, and we will study it by its special representation defined above for an appropriate $f$ (depending of course on $r$). For appropriate choices of $r, \alpha$, we will see that these systems are weakly mixing, and hence so are the special flows. This will be an important component of the proof of the main theorem.

It is also important to note that reparametrized flows on the 2-Torus are uniquely ergodic, meaning there is only one Borel probability measure $\nu$ on $\T^2$ such that the flow is ergodic.

If in addition the flow is weak mixing (which can be arranged, as we will see in section 3.6), this implies that the time-1 map $T_1^{\alpha, r}$ is uniquely ergodic: suppose $\lambda \neq \nu$ is an invariant measure for  $T_1^{\alpha, r}$, then for every $s \in [0,1]$ the measure $(T^{\alpha,r}_s)_* \lambda$ is also invariant for $T_1$, and the measure $\int_0^1 (T^{\alpha,r}_s)_* \lambda ds$ is $T_t^{\alpha, r}$ invariant. Therefore we must have $\nu = \int_0^1 (T^{\alpha,r}_s)_* \lambda ds$. Since $T^{\alpha,r}_t$ is weak mixing, $T_1^{\alpha, r}$ is ergodic, and so $(T^{\alpha,r}_s)_* \lambda = \nu$ for almost every $s \in [0,1]$. Therefore, we have the desired result, that $\nu = (T^{\alpha,r}_{-s})_* \nu = \lambda$, and so $T_1^{\alpha, r}$ is uniquely ergodic.

 Now since $T_1^{\alpha, r}$ is uniquely ergodic (whenever $T_t^{\alpha, r}$ is weak mixing), we get the following: for any $f \in C(\T^2)$, we have
$$
\frac{1}{n} S_n(f)(x) \to \int_{\T^2} f d\nu,
$$
uniformly in $x$ (where the underlying transformation is $T_1^{\alpha, r}$).

The key point for us is that for the reparameterized linear flow, every point's orbit is equidistributed. We will show that there is some structure to orbits along particular primes and then apply this result to sums along residue classes modulo these primes.

Our main result is stated in the language of smooth reparameterizations of linear flows on $\T^2$, but the proofs will be conducted in the special representation model. We will see in the following section that these two models are smoothly conjugated.

\subsection{Conjugacy Between the Flows}
For completeness, we recall the construction of the conjugacy between reparameterizations of smooth flows on $\T^2$ and their special representation. Let $(T^{\alpha,r}_t)$ be a smooth reparametrized flow on $\T^2$. The basic idea is to define the functional values for $f: \T \to \R^+$ for our suspension flow by integrating $r$ along flow lines in $\T^2$: Let $C_x$ be the path given by flowing from the point $(x,0) \in \T^2$ to the point $(x+\alpha, 1) \in \T^2$ in the linear flow given by $\alpha, r$. We then define $f: \T \to \R^+$ by $f(x) = \int_{C_x} r(x(t), y(t)) dt$, i.e. $f(x)$ will essentially be the total time required to flow from $(x,0)$ to $(x+\alpha, 1)$ in $\T^2$. Since we will be flowing with unit speed in the corresponding special flow, this is really just a way to again reparameterize the original flow on $\T^2$. There is then an immediate conjugacy between the two systems $(\T^2, T^{\alpha,r}_t)$ and $(\T_f, T_t^f)$.

\subsection{Weak Mixing}
This is a result of Fayad \cite{Fayad}: for any given irrational $\alpha$, there is an analytic function $r: \T^2 \to \R^+$ such that the corresponding linear flow on $\T^2$ is weak mixing. His result is actually stronger than this (and holds in higher dimensions as well), but this is all that will be necessary for us. It is also important to note that rescaling such $r$ by a constant will not change the fact that the linear flow on $\T^2$ is weak-mixing (this will allow us to assume later that the corresponding $f$ for the special flow has an integral of $1$).

\subsection{The Siegel-Walfisz and Prime Number Theorems}
These results from number theory will be key ingredients in the proof of the main theorem. Before stating the results, we introduce some notation. 

Let $\pi (x; q, a)$ denote the number of primes less than or equal to $x$ which are congruent to $a \pmod q$, and $\pi(x)$ be the prime counting function (giving the number of primes less than or equal to $x$). We denote Euler's totient function by $\phi$ (meaning that $\phi(n)$ is the number of positive integers $\leq n$ which are relatively prime to $n$), and use the notation $Li(x) = \int_2^x \frac{dt}{\log(t)}$ for the offset logarithmic integral. We also use the following asymptotic notation: $f(x) = O(g(x))$ if there is a constant $M$ such that $\lvert f(x) \rvert \leq M g(x)$ for all sufficiently large $x$, and $f(x) \sim g(x)$ if $\lim_{x \to \infty} \frac{f(x)}{g(x)} = 1$.

Now we are ready to state the results.

\begin{theorem}[Siegel-Walfisz]
Let $A \in \R$ be given, and assume $q \leq (\log x)^A$, $\gcd(q,a) = 1$. Then there exists $C_A$, depending only on $A$, such that
$$
\pi (x; q, a) = \frac{\text{Li}(x)}{\phi(q)} + O \left(x \cdot e^{\frac{-C_A}{2} (\log x)^{1/2}} \right).
$$
\end{theorem}

We get the following prime number theorem for arithmetic progressions as a corollary to the above (by setting $A=2$):

\begin{corollary}\label{cor:SW}
Uniformly in all primes $q < (\log x)^2$ and $0 < a < q$ (with $\gcd(q,a)=1$),
$$
\pi(x; q, a) \sim \frac{Li(x)}{q-1}.
$$
\end{corollary}

Now setting $q=2$ and $a=1$ in the above Corollary, since $\pi(x;2,1) = \pi(x)$, we obtain the classical Prime Number Theorem.

\begin{theorem}[The Prime Number Theorem] Let $\pi(x)$ denote the prime counting function, and $Li(x)$ denote the offset logarithmic integral. Then
$$
\pi(x) \sim Li(x).
$$
\end{theorem}

\subsection{The Set of Irrationals}

It will be necessary to impose some Diophantine conditions on the irrational $\alpha$ used to define the flows described above. Let $c_o, \: \delta$ be positive constants. Define
$$
D = \{ \alpha \in \R \setminus \Q \mid \forall n \in \N, \: q_n \text{ is prime (or $1$), and } q_{n+1} \geq c_o e^{\delta q_n} \}.
$$
For simplicity, from now on we will assume that $\alpha \in D$. However, the methods applied in this paper may also be used if the above conditions hold only along a subsequence of the denominators $(q_n)$. 

The following lemma is standard, we provide a proof for completeness.

\begin{lemma}
The set $D$ is uncountable.
\end{lemma}

\begin{proof}
We will form an injection from the set $\{0,1\}^\N$, which is uncountable, to the set of $D$'s continued fractions. Let $x = (x_1, x_2, ...) \in \{0,1\}^\N$. Set $q_0 = q_1 = 1$. Note that by Dirichlet's Theorem on Arithmetic progressions, the set $\{ m \cdot q_1 + q_0 \}_{m \in \N}$ contains infinitely many primes (since $\gcd(q_1, q_0) = 1$), and thus primes which are arbitrarily large. We can now define $q_2$: if $x_1 = 0$, let $a_1$ be the smallest natural number such that $q_2 = a_1 \cdot q_1 + q_0$ is prime and larger than $c_o e^{\delta q_1}$; if $x_1 = 1$, let $a_1$ be the \it second \normalfont smallest natural number so that $q_2$ has these properties. We now have that $\gcd(q_2, q_1) = 1$, so we can apply Dirichlet's Theorem again to obtain $q_3$ in an analogous manner (dependent on if $x_2 = 0 \text{ or } 1$).

It is easy to prove by induction that at each step this construction yields numbers $q_n, q_{n-1}$ for which $\gcd(q_n, q_{n-1}) = 1$, so that by Dirichlet's Theorem, the sequence $\{ m \cdot q_n + q_{n-1} \}_{m \in \N}$ always contains primes that are arbitrarily large. Thus in general, at the $n^{\text{th}}$ iterate, if $x_n = 0$, $a_n$ will be the smallest natural number such that $q_{n+1} = a_n \cdot q_n + q_{n-1}$ is prime and larger than $c_o e^{\delta q_n}$, and if $x_n = 1$, it will be the second smallest natural number with these properties. So for any given $x \in \{0,1\}^\N$, we can associate a unique infinite sequence $(a_n)$ of natural numbers, which is uniquely associated to some irrational $\alpha \in D$ by the continued fraction expansion. Hence the set $D$ is uncountable.

\end{proof}

\section{Ergodic Averages of Analytic Cocycles over $\T$}

From this point on, we will denote the special flow on $\T_f$ by $(T_t)$ since there will be no ambiguity regarding the underlying function $f$ determining the space. We will denote the characters of $\T$ by

$$
\chi_k (x) = e^{2 \pi i k x}.
$$
Any function $f \in C(\T)$ can be written as a linear combination of characters
\begin{equation}\label{eq:sum}
f = \sum_{k \in \Z} a_k \chi_k ,
\end{equation}
(where the $(a_k)_{k \in \Z}$ are constants) and further $f$ is a \it uniform \normalfont limit of these sums.

Recall that by the formula for finite geometric sums we have

\begin{equation}\label{eq:geomsum}
S_{q_n}(\chi_k)(x)= e^{2 \pi i k x} \cdot \frac{1-e^{2 \pi i k q_n \alpha}}{1-e^{2 \pi i k \alpha}}.
\end{equation}

We first establish some natural bounds on the size of the ergodic sums $S_{q_n}(f)(x)$.

\begin{lemma}\label{lem:form}
For every $k \in \Z$, we have that
$$
\lvert S_{q_n}(\chi_k)(x) \rvert \leq \text{min}\left(q_n, \, \frac{4 \pi \lvert k \rvert}{\Vert k \alpha \Vert \cdot q_{n+1}} \right).
$$
\end{lemma}

\begin{proof}
First, observe that
$$
\lvert S_{q_n}(\chi_k)(x) \rvert = \biggr\lvert \sum_{j=0}^{q_n - 1} e^{2 \pi i k (x+j\alpha)} \biggr\rvert \leq \sum_{j=0}^{q_n - 1} \lvert e^{2 \pi i k (x+j\alpha)} \rvert = \sum_{j=0}^{q_n- 1} 1 = q_n,
$$
by the triangle inequality. Thus it just remains to show that
$$
\lvert S_{q_n}(\chi_k)(x) \rvert \leq \frac{4 \pi \lvert k \rvert}{\Vert k \alpha \Vert \cdot q_{n+1}}.
$$
By \eqref{eq:geomsum}, we have that
$$
\lvert S_{q_n}(\chi_k)(x) \rvert = \frac{\lvert 1-e^{2 \pi i k q_n \alpha}\rvert}{\lvert 1-e^{2 \pi i k \alpha}\rvert}.
$$
But now observe that $\chi_k \in Lip(\mathbb{T})$: let $a \leq b$,
$$
\lvert e^{ib} - e^{ia} \rvert = \biggr\lvert \int_a^b ie^{it} dt \biggr\rvert \leq \int_a^b \lvert ie^{it} \rvert dt = \int_a^b 1 dt = b - a,
$$
hence we have that
$$
\lvert e^{2 \pi i k x} - e^{2 \pi i k y} \rvert \leq \lvert 2 \pi k x - 2 \pi k y \rvert = 2 \pi \lvert k \rvert \cdot \lvert x - y \rvert,
$$
which is precisely the statement that $\chi_k$ is Lipschitz with Lipschitz constant $2 \pi \lvert k \rvert$.
Now, applying \eqref{eq:norm}, we have
$$
\lvert e^{2 \pi i k (q_n \alpha)} - 1 \rvert = \lvert e^{2 \pi i k (q_n \alpha)} - e^{2 \pi i k (p_n)} \rvert \leq 2 \pi \lvert k \rvert \cdot \lvert q_n \alpha - p_n \rvert \leq \frac{2 \pi \lvert k \rvert}{q_{n+1}}.
$$
Hence,
\begin{equation}\label{eq:hence}
\lvert S_{q_n}(\chi_k)(x) \rvert \leq \frac{1}{\lvert 1-e^{2 \pi i k \alpha}\rvert} \cdot \frac{2 \pi \lvert k \rvert}{q_{n+1}}.
\end{equation}
Note that the Taylor series of $e^x$ is $1 + x + \frac{x^2}{2!} + ...$ . Let $m_k$ denote the integer that is closest to $k\alpha$. Then
\begin{align*}
\lvert 1 - e^{2 \pi i k \alpha} \rvert &= \lvert 1 - e^{2 \pi i \lvert k \alpha - m_k\rvert } \rvert \\
&= \biggr\lvert 1 - (1 + 2 \pi i \lvert k \alpha - m_k \rvert + \frac{[2 \pi i \lvert k \alpha - m_k \rvert]^2}{2!} + ...) \biggr\rvert \geq \frac{\lvert k \alpha - m_k \rvert}{2} = \frac{\Vert k \alpha \Vert}{2},
\end{align*}
and so by \eqref{eq:hence}, we have our desired bound
$$
\lvert S_{q_n}(\chi_k)(x) \rvert \leq \frac{1}{\Vert k \alpha \Vert } \cdot \frac{4 \pi \lvert k \rvert}{q_{n+1}}.
$$
\end{proof}

Now, recall that we can decompose $f$ as a given in \eqref{eq:sum}, and hence

\begin{equation}\label{eq:birkhoffsum}
    S_{q_n}(f)(x) = \sum_{k = -\infty}^{\infty} a_k \cdot S_{q_n}(\chi_k)(x),
\end{equation}

by additivity of Birkhoff sums.

\begin{corollary}
For any $f \in C(\T)$,
$$
\lvert S_{q_n}(f)(x) \rvert \leq \sum_{k = -\infty}^{\infty} a_k \cdot \text{min}\left(q_n, \, \frac{4 \pi \lvert k \rvert}{\Vert k \alpha \Vert \cdot q_{n+1}} \right).
$$
\end{corollary}

\begin{proof}
Apply Lemma \ref{lem:form} and \eqref{eq:birkhoffsum} (making use of the triangle inequality).
\end{proof}

\begin{lemma}\label{lem:bound}
If $f$ is analytic, then
$$
\lvert S_{q_n}(f)(x) \rvert \leq C\sum_{k = -\infty}^{\infty} e^{-c \lvert k \rvert} \cdot \text{min}\left(q_n, \, \frac{4 \pi}{\Vert k \alpha \Vert \cdot q_{n+1}} \right),
$$
for some positive constants C, c.
\end{lemma}

\begin{proof}
For analytic functions on the torus, the Fourier coefficients obey the following bound: $\lvert a_k \rvert \leq C' e^{-c' \lvert k \rvert}$ for some positive constants $C', c'$. Thus it only remains to show that for some other constants $C, c$, $e^{-c' \lvert k \rvert} \cdot \lvert k \rvert \leq Ce^{-c \lvert k \rvert}$, but now observe that
$$
e^{-c' \lvert k \rvert} \cdot \lvert k \rvert \leq Ce^{-c \lvert k \rvert} \; \Leftrightarrow \; \lvert k \rvert \leq Ce^{(c' - c)\lvert k \rvert},
$$
which is clearly satisfied by large enough $C$ as long as $c<c'$.
\end{proof}

Now, using the previous estimates, we move to establish bounds on the quantities $\lvert S_{Kq_n}(f)(x) - Kq_n \int_\mathbb{T} f d\mu\rvert$ in the spirit of the Denjoy-Koksma inequality. The purpose of this is that later we will assume $\int_\T f d\mu = 1$, so that we have bounds on the distance of these ergodic sums to integer multiples of the $q_n$.

\begin{lemma}\label{lem:sum-int}
Suppose $f$ is analytic. For sufficiently large n and some constants $C', c'$,
$$
\biggr\lvert S_{q_n}(f)(x) - q_n \int_\mathbb{T} f d\mu \biggr\rvert \leq C' e^{-c' q_n}.
$$
\end{lemma}

\begin{proof}
First note that for every $x\in \T$, $S_{q_n}(\chi_0)(x)=q_n \int_\mathbb{T} f d\mu$, so, by \eqref{eq:birkhoffsum} it suffices to show that 
$$
\biggr \lvert \sum_{k\neq 0}a_kS_{q_n}(\chi_k)(x) \biggr \rvert \leq C' e^{-c'q_n},
$$
for some constant $c'>0$. By Lemma \ref{lem:bound}, we have
$$
\biggr \lvert \sum_{k\neq 0}a_kS_{q_n}(\chi_k)(x) \biggr \rvert \leq \sum_{0 < |k| < q_n} \frac{4 \pi e^{-c \lvert k \rvert}}{\Vert k \alpha \Vert q_{n+1}} + \sum_{|k| \geq q_n} e^{-c \lvert k \rvert} q_n,
$$
and now using \eqref{eq:norm} and that $(p_n/q_n)$ is the sequence of best approximations of $\alpha$,
$$
\frac{1}{2q_n} \leq \Vert q_{n-1} \alpha \Vert \leq \Vert k \alpha \Vert,
$$
so we have that
$$
\biggr \lvert \sum_{k\neq 0}a_kS_{q_n}(\chi_k)(x) \biggr \rvert \leq \sum_{0 < |k| < q_n} \frac{4 \pi e^{-c \lvert k \rvert}}{\frac{1}{2q_n} q_{n+1}} + q_n \sum_{|k| \geq q_n} e^{-c \lvert k \rvert} = \frac{8 \pi q_n}{q_{n+1}} \sum_{0 < |k| < q_n}  e^{-c \lvert k \rvert} + q_n \sum_{|k| \geq q_n} e^{-c \lvert k \rvert}.
$$
So by the formula for summing geometric series, we have
$$
\biggr \lvert \sum_{k\neq 0}a_kS_{q_n}(\chi_k)(x) \biggr \rvert \leq \frac{16 \pi q_n}{q_{n+1}} \left( \frac{1 - e^{-c q_n}}{1 - e^{-c}} \right) + 2 q_n \frac{e^{-c q_n}}{1 - e^{-c}}.
$$
Now, since $q_{n+1} \geq c_o e^{\delta q_n}$, we have
\begin{align*}
\biggr \lvert \sum_{k\neq 0}a_kS_{q_n}(\chi_k)(x) \biggr \rvert &\leq \frac{16 \pi}{c_o (1 - e^{-c})} q_n e^{-\delta q_n} - \frac{16 \pi}{c_o (1 - e^{-c})} q_n e^{(-c-\delta)q_n} + \frac{2}{1 - e^{-c}} q_n e^{-cq_n} \\ 
&\leq M q_n e^{-d q_n},
\end{align*}
for some positive constants $M, d$, hence
$$
\biggr \lvert \sum_{k\neq 0}a_kS_{q_n}(\chi_k)(x) \biggr \rvert \leq C' e^{-c'q_n},
$$
for some positive constants $C', c'$, which concludes the proof.
\end{proof}

\begin{lemma}\label{lem:uniform}
Let $f$ be analytic. Then for all $d \in (0,c')$,
$$
\sup_{0 \leq K \leq e^{dq_n}} \; \sup_{x\in\mathbb{T}} \; \biggr\lvert S_{Kq_n}(f)(x) - Kq_n \int_\mathbb{T} f d\mu \biggr\rvert \to 0,
$$
as $n\to +\infty$.
\end{lemma}

\begin{proof}
First, by the cocycle identity, we have
$$
S_{Kq_n}(f)(x) = \sum_{j=0}^{K-1} S_{q_n}(f)(T^{j q_n} x),
$$
and so we have
\begin{align*}
\biggr\lvert S_{Kq_n}(f)(x) - Kq_n \int_\mathbb{T} f d\mu \biggr\rvert &= \biggr\lvert \sum_{j=0}^{K-1} \left( S_{q_n}(f)(T^{j q_n} x) - q_n \int_\mathbb{T} f d\mu \right) \biggr\rvert \\
&\leq \sum_{j=0}^{K-1} \biggr\lvert S_{q_n}(f)(T^{j q_n} x) - q_n \int_\mathbb{T} f d\mu \biggr\rvert.
\end{align*}
Thus by Lemma \ref{lem:sum-int}, for any $x \in \mathbb{T}$,
$$
\biggr\lvert S_{Kq_n}(f)(x) - Kq_n \int_\mathbb{T} f d\mu \biggr\rvert \leq \sum_{j=0}^{K-1} e^{-c' q_n} = Ke^{-c' q_n},
$$
so let $d < c'$ be given, and assume $0 \leq K \leq e^{d q_n}$, then we have
$$
\biggr\lvert S_{Kq_n}(f)(x) - Kq_n \int_\mathbb{T} f d\mu \biggr\rvert \leq e^{(d - c') q_n}.
$$
We then have that
$$
\sup_{0 \leq K \leq e^{dq_n}} \; \sup_{x\in\mathbb{T}} \; \biggr\lvert S_{Kq_n}(f)(x) - Kq_n \int_\mathbb{T} f d\mu \biggr\rvert \leq e^{(d - c') q_n},
$$
and so taking the limit as $n \to +\infty$, we have the result.
\end{proof}

\section{Proof of the Main Theorem}

Let $m$ denote the metric induced on $\T_f$ realized as a subset of $\R^2$. Recall also the relevance of the constant $\delta$, as used in the definition of the set $D$, and the constant $c'$, from the statements of Lemmas \ref{lem:sum-int} and \ref{lem:uniform}.

\begin{lemma}\label{lem:close}
Suppose $f$ is analytic. Then for all $d < \min (\delta, c')$, we have
$$
\sup_{0 \leq K \leq e^{dq_n}} \; \sup_{x\in\mathbb{T}} \; m(T_{S_{K q_n}(f)(x)} (x, s), (x,s)) \to 0,
$$
as $n \to +\infty$.
\end{lemma}

\begin{proof}
Let $(x,s)$ be given such that $0 < s < f(x)$, and let $0 \leq K \leq e^{d q_n}$. Let $n$ be large enough that
\begin{equation}\label{eq:height}
\forall x' \in B_{e^{(d - \delta)q_n}}(x), \, s < f(x').
\end{equation}
Recall that the definition of our flow $(T_t)$ is the following:
$$
T_t(x,s) = (T^n x, s'),
$$
where $n, s$ satisfy
\begin{equation}\label{eq:satisfy}
S_n(f)(x) + s' = t + s, \text{ and } 0 \leq s' \leq f(T^n x).
\end{equation}
But now observe that
$$
T^{Kq_n}x = x + Kq_n \alpha \pmod{1}.
$$
Now, since $\Vert K q_n \alpha \Vert \leq \frac{K}{q_{n+1}}$, we have
$$
\lvert x - T^{Kq_n}x \rvert \leq \frac{K}{q_{n+1}} \leq \frac{e^{dq_n}}{e^{\delta q_n}} \leq e^{(d - \delta) q_n}.
$$
Therefore, since we have $T^{Kq_n}x \in  B_{e^{(d - \delta)q_n}}(x)$, set $t = S_{Kq_n}(f)(x)$. By applying \eqref{eq:height}, in order to satisfy \eqref{eq:satisfy} we must have that
$$
n = K q_n \text{ and } s' = s.
$$

So now since $s = s'$, taking the limit as $n \to +\infty$ we obtain the result for all points strictly inside (not on the boundary) of $\T_f$.

In the special case that $(x,s)$ is on the boundary, consider the point with which it is identified such that $s = 0$ (either the point given if $s = 0$ or the point $(Tx, 0)$ if $s = f(x)$). Flowing for time $S_{K q_n}(f)(x)$ then corresponds precisely to rotating by $\alpha$ on $\T$ exactly $Kq_n$ times, and the same argument works.
\end{proof}

\begin{lemma}\label{lem:closer}
Assume $f$ is analytic and that $\int_\T f d\mu = 1$. For all $d < \min (\delta, c')$, we have
$$
\sup_{0 \leq K \leq e^{dq_n}} \; \sup_{x\in\mathbb{T}} \; m(T_{S_{K q_n}(f)(x)} (x, s), T_{Kq_n}(x,s)) \to 0,
$$
as $n \to +\infty$
\end{lemma}

\begin{proof}
This follows from the continuity of the flow $(T_t)$ and Lemma \ref{lem:uniform}.
\end{proof}

\begin{corollary}\label{cor:closest}
Assume $f$ is analytic and $\int_\T f d\mu = 1$. For all $d < \min (\delta, c')$, we have
$$
\sup_{0 \leq K \leq e^{dq_n}} \; \sup_{x\in\mathbb{T}} \; m(T_{Kq_n} (x, s), (x,s)) \to 0,
$$
as $n \to +\infty$
\end{corollary}

\begin{proof}
Apply the triangle inequality along with Lemmas \ref{lem:close} and \ref{lem:closer}.
\end{proof}

We now have all the machinery required to prove the main theorem.

\begin{theorem}
Assume $f$ is analytic and $\int_\T f d\mu = 1$. Let $d<\min (\delta, c')$ be given. Define the sequence $(K_n)$ by $K_n = e^{d q_n}$. Then for any $(x,s) \in \T_f$, the sequence $(T_p(x,s))_{p < K_n q_n}$ is equidistributed.
\end{theorem}

\begin{proof}
Let $g \in C(\T_f)$ be given. We must show that
$$
\frac{1}{\pi(K_n q_n)} \sum_{p < K_n q_n} g(T_p(x,s)) \to \int_{\T_f} g d\nu.
$$
First, we will rewrite this sum as a double sum along residue classes modulo $q_n$; for $p \equiv a \pmod{q_n}$, we will write $p = k_p q_n + a$:
$$
\sum_{p < K_n q_n} g(T_p(x,s)) = \sum_{a < q_n} \left( \sum_{\substack{p = k_p q_n + a \\ p < K_n q_n}} g(T_p(x,s)) \right) = \sum_{a < q_n} \left( \sum_{\substack{p = k_p q_n + a \\ p < K_n q_n}} g(T_{k_p q_n + a}(x,s)) \right).
$$
Now, since
$$
T_{k_p q_n + a}(x,s) = T_{k_p q_n}(T_a(x,s)) \to T_a(x,s),
$$
uniformly in $(x,s)$ as $n \to +\infty$ (by Corollary \ref{cor:closest}), we see that
\begin{equation}\label{eq:bigone}
\frac{1}{\pi(K_n q_n)} \biggr\lvert \sum_{p < K_n q_n} g(T_p(x,s)) - \sum_{a < q_n} \left( \sum_{\substack{p = k_p q_n + a \\ p < K_n q_n}} g(T_a(x,s)) \right) \biggr\rvert \to 0,
\end{equation}
uniformly in $(x,s)$ (applying continuity of $g$). Now we may rewrite the double sum
$$
\sum_{a < q_n} \left( \sum_{\substack{p = k_p q_n + a \\ p < K_n q_n}} g(T_a(x,s)) \right) = \sum_{a < q_n} g(T_a(x,s)) \cdot \pi(K_n q_n; q_n, a).
$$
So, replacing the double sum in \eqref{eq:bigone} using the previous identity and distributing the outside factor, we get
$$
\biggr\lvert \frac{1}{\pi(K_n q_n)} \sum_{p < K_n q_n} g(T_p(x,s)) - \frac{1}{\pi(K_n q_n)} \sum_{a < q_n} g(T_a(x,s)) \cdot \pi(K_n q_n; q_n, a) \biggr\rvert \to 0.
$$
Finally, for the second term in this difference, we apply the Prime Number Theorem on $\pi(K_n q_n)$ (to replace it with $Li(K_n q_n)$) and Corollary \ref{cor:SW} on $\pi(K_n q_n; q_n, a)$ (to replace it with $Li(K_n q_n)/(q_n-1)$). We obtain
$$
\biggr\lvert \frac{1}{\pi(K_n q_n)} \sum_{p < K_n q_n} g(T_p(x,s)) - \frac{1}{Li(K_n q_n)} \sum_{a < q_n} g(T_a(x,s)) \cdot \frac{Li(K_n q_n)}{q_n-1} \biggr\rvert \to 0,
$$
hence, cancelling, we have
$$
\biggr\lvert \frac{1}{\pi(K_n q_n)} \sum_{p < K_n q_n} g(T_p(x,s)) - \frac{1}{q_n-1} \sum_{a< q_n} g(T_a(x,s)) \biggr\rvert \to 0.
$$
Now, since $T_1$ is uniquely ergodic, the sequence $(T_n(x,s))_{n \in \N}$ is equidistributed for all $(x,s) \in \T_f$, meaning precisely that
$$
\frac{1}{q_n-1} \sum_{a< q_n} g(T_a(x,s)) \to \int_{\T_f} g d\nu,
$$
and therefore
$$
\frac{1}{\pi(K_n q_n)} \sum_{p < K_n q_n} g(T_p(x,s)) \to \int_{\T_f} g d\nu,
$$
which completes the proof.
\end{proof}

\begin{corollary}
For any $(x,s) \in \T_f$, the sequence $(T_p(x,s))$ is dense in $\T_f$.
\end{corollary}

\end{document}